\documentclass[12pt]{amsart}
\usepackage{graphicx}
\newtheorem{thm}{Theorem}[section]
\newtheorem{cor}[thm]{Corollary}

\newtheorem{prop}[thm]{Proposition}

\newtheorem{conj}[thm]{Conjecture}
\theoremstyle{definition}
\newtheorem{defn}[thm]{Definition}
\theoremstyle{remark}
\newtheorem{rem}[thm]{Remark}
\numberwithin{equation}{section}
\def\cp{\mathcal{ P}}

\def\gb{{\frak B}}

\def\bn{{\mathbb N}}

\def\br{{\mathbb R}}

  \def\G{\Gamma}

\def\xb{{\mathbf{x}}}
\def\eb{{\mathbf{e}}}
\def\yb{{\mathbf{y}}}
\def\zb{{\mathbf{z}}}

\def\yb{{\mathbf{y}}}

\def\ib{{\mathbf{i}}}

\begin{document}
\title
{On Non-Linear Markov Operators: surjectivity vs orthogonal
preserving property}

\author{Farrukh Mukhamedov}
\address{Farrukh Mukhamedov\\
 Department of Mathematical Sciences, College of Science, The United Arab Emirates University, P.O. Box 15551, Al Ain, Abu Dhabi, UAE} \email{{\tt
farrukh.m@uaeu.ac.ae}}

\author{Ahmad Fadillah Embong}
\address{Ahmad Fadillah Embong\\
Department of Computational \& Theoretical Sciences, Faculty of
Science, International Islamic University Malaysia, 25200, Kuantan,
Pahang, Malaysia} \email{{\tt ahmadfadillah.90@gmail.com}}

\maketitle

\begin{abstract}
In the present paper, we consider nonlinear Markov operators, namely
polynomial stochastic operators. We introduce a notion of orthogonal
preserving polynomial stochastic operators. The purpose of this
study is to show that surjectivity of nonlinear Markov operators is
equivalent to their orthogonal preserving property. \noindent {\it
Mathematics Subject Classification}: 47H25, 37A30,
47H60

%a class of CSO which is orthogonal preserving CSOs (OP CSOs). Here we are aiming to describe surjectivity of CSO by mean of OP CSO. It is proven that surjectivity of CSO is equivalent to OP CSO. In the last section provide full descriptions and non-conjugate classes of OP CSO (or surjective CSO) on two dimensional simplex.
\end{abstract}

\section{Introduction}

Recently, nonlinear Markov chains are intensively studied by many
scientists (see \cite{Kol} for recent review). A process described
by a nonlinear Markov operator is a discrete time stochastic process
whose transitions may depend on both the current state and the
present distribution of the process. The power of nonlinear Markov
operator as a modeling tool and its range of applications are
immense, and include non-equilibrium statistical mechanics,
evolutionary biology (replicator dynamics), population and disease
dynamics (Lotka-Volterra and epidemic models) and the dynamics of
economic and social systems (replicator dynamics and games).

 The simplest nonlinear Markov chain is described by a quadratic stochastic operator (QSO)
which is associated with a cubic stochastic matrix. This kind of
operator arises in the problem of describing the evolution of
biological populations \cite{11}. The notion of QSO was first
introduced by Bernstein
 \cite{1} and the theory of QSOs was developed in many works (see for example \cite{K,11,32}).
In \cite{6,MG2015book}, it is given along self-contained exposition
of the recent achievements and open problems in the theory of the
QSOs.

Letting $ I_{m} = \{ 1, \dots, m \} $, a straightforward calculation
shows that if a stochastic linear operator $ L : S^{m-1} \rightarrow
S^{m-1} $ is surjective (here $S^{m-1}$ is the set of all
probability distributions on $I_m$), then, for each $ i \in I_{m} $,
there exists a $ j \in I_{m} $ such that $ L^{-1}(\eb_{i}) = \eb_{j}
$, where
 $ L^{-1}(\eb_{i}) $ is the preimage of the vertex $ \eb_{i} $ of the simplex $ S^{m-1} $.
 Unfortunately, this is not the case when we consider nonlinear
 case. On the other hand, the surjectivity
 of a nonlinear operator is strongly tied up with nonlinear optimization problems \cite{Arut2012}.
 The criteria for the
 surjectivity of QSOs was given in \cite{man(2016)2d}. The obtained
 criteria together with results of \cite{taha} implies that a QSO is surjective if and only if it is orthogonal preserving.
In \cite{MFR} we have check this property for cubic stochastic
operators, and described all surjective cubic stochastic operators
on two-dimensional simples. Hence, it is natural to study the same
implications for general non-linear Markov operators.

In this paper, we introduce a notion of orthogonal preservness for
nonlinear Markov operators, and show that the surjectivity of this
kind of operators is equivalent to their orthogonal preserving
property.

\section{Preliminaries}

Let us recall some necessary notations. Let $ I_{m} = \{ 1, \dots, m
\} $. The complement of a set $ A \subset I_{m} $ is denoted by $
A^{c} = I_{m} \backslash A $. By $\{\eb_i\}_{i\in I_m}$ we denote
the standard basis in $\br^m$. Throughout this paper we consider the
simplex as
\begin{eqnarray}\label{kthresidue}
S^{m-1} = \left\{\textbf{x}\in\br^m\ : \ x_i\geq0, \ \forall i\in
I_m, \quad \sum\limits_{i=1}^{m}x_i = 1\right\} \label{eqn1.1}.
\end{eqnarray}
An element of the simplex $S^{m-1}$ is called a \textit{stochastic
vector}.

For a every $ \xb \in S^{m-1} $ we set
$$ supp(\xb) = \left\{ i \in
I_{m}\ :  x_{i} \neq 0 \right\}, \ \  null(\xb) = \{ i\in \bn \ :
x_{i}=0 \}.$$ We define the facet $ \G_{A} $ of the simplex $
S^{m-1} $ by setting $ \G_{A} = conv\{\eb_{i} \}_{i \in A} $, here
$conv( B )$ stands for the convex hull of a set $ B $. Let
\[ int \G_{A} = \{ \xb \in \G : x_{i}>0,\ \forall \ i \in A \} \]
be the relative interior of $ \G_{A} $.

Let $\cp = (P_{i_1\dots,i_k} )^{m}_{i_1,...,i_k=1}$ be a $k$-order
$m$-dimensional hypermatrix. We define the following vectors
$$
P_{i_{1}\cdots i_{l},\bullet} = (P_{i_{1}\cdots i_{l},1}, \dots,
P_{i_{1}\cdots i_{l},m}), $$ for any $i_1,\dots, i_k\in I_m$.
 For the sake of simplicity we use $
i_{[1,l]} = i_{1},\dots,i_{l} $ for index. A hypermatrix $
\mathcal{P} $ is called \textit{stochastic} if  each vector $
P_{i_{[1,l]},\bullet} $ is stochastic for any $ i_{1},\dots, i_{l}
\in I_{m} $.

Let $\cp$ be a stochastic hypermatrix, then it defines a
\textit{nonlinear Markov operator} (or \textit{polynomial stochastic
operator (PSO)}) $ \mathfrak{B} : S^{m-1} \rightarrow S^{m-1} $ as
follows
\[ \mathfrak{B}(\xb) = \sum\limits_{i_{1}=1}^{m} \cdots \sum\limits_{i_{l}=1}^{m}x_{i_{1}} \cdots x_{i_{l}} P_{i_{[i,l]},\bullet}, \ \ \xb\in S^{m-1}.  \]
Throughout this paper, without loss of generality, we assume that
\begin{eqnarray}\label{eqn_prop_P}
P_{i_{[1,l]},\bullet} =  P_{{\pi} (i_{[1,l]}),\bullet}
\end{eqnarray}
for any $i_1,\dots, i_l\in I_m$ and any permutation ${\pi}$ of the
indices.

\begin{rem} We notice that if the hypermatrix $\cp$ is given by the
cubic matrix $(p_{ij,k})$, then the associated PSO reduces to the
\textit{quadratic stochastic operator (QSO)} given by
\begin{eqnarray*}
 V(\xb)_k=\sum\limits_{i,j=1}^m p_{ij,k}x_ix_j, \ \ k\in I_m.
\end{eqnarray*}
\end{rem}

\begin{rem} We stress that a PSO $\gb_0$ associated with a stochastic hypermatrix $\cp$ of the
$\ell-1$-order can be considered as a particular case of PSO
associated with a stochastic hypermatrix of $\ell$-order. Indeed,
assume that $(P_{i_{[1,i_{\ell-1}]}})$ is an $\ell-1$-order
stochastic hypermatrix. Now define $\ell$-order hypermatrix by
$$
\tilde
P_{i_1,\dots,i_{\ell-1},k}=\frac{1}{n}\sum_{\textrm{combinations} \
\{i_s\}_{s=1}^{\ell-1} \atop \textrm{over}\  \ell-2 \
\textrm{places}}P_{i_1,\dots,i_{\ell-2},k}
$$
Then direct calculations show that
\begin{eqnarray*}
\mathfrak{B}(\xb)&=&\sum\limits_{i_{1}=1}^{m} \cdots
\sum\limits_{i_{\ell-1}=1}^{m}x_{i_{1}} \cdots x_{i_{\ell-1}} \tilde
P_{i_{[i,\ell-1]},\bullet}\\[2mm]
&=&\sum\limits_{i_{1}=1}^{m} \cdots
\sum\limits_{i_{\ell-1}=1}^{m}x_{i_{1}} \cdots x_{i_{\ell-1}}\bigg(
\frac{1}{n}\sum_{\textrm{combinations} \ \{i_s\}_{s=1}^{\ell-1}
\atop \textrm{over}\  \ell-2 \
\textrm{places}}P_{i_1,\dots,i_{\ell-2},\bullet}\bigg)\\[2mm]
&=&\sum\limits_{i_{1}=1}^{m} \cdots
\sum\limits_{i_{\ell-2}=1}^{m}x_{i_{1}} \cdots
x_{i_{\ell-2}}P_{i_1,\dots,i_{\ell-2},\bullet}\\[2mm]
&=&\gb_0(\xb).
\end{eqnarray*}
This means that $\gb_0$ is a particular case of $\gb$. Using the
provided technique, one can show that any QSO is also particular
case of PSO. Therefore, methods used for QSO may not be valid in
general setting.
\end{rem}

We recall $ \xb \in S^{m-1}$ is \textit{orthogonal} or
\textit{singular} to $ \yb \in S^{m-1}$ ($ \xb \perp \yb $) if and
only if $ supp(\xb) \cap supp(\yb) = \emptyset $.
 It is clear that $ \xb \perp \yb $ if and only if $ x_{k} \cdot y_{k} =0 $ for all $ k \in  I_{m} $ whenever $ \xb, \yb \in S^{m-1} $.

\begin{defn}
A PSO $ \gb $ is called \textit{orthogonal preserving (OP PSO)} if
for any $ \xb,\yb \in S^{m-1} $ with $ \xb \perp \yb $ implies $
\gb(\xb) \perp \gb(\yb) $.
\end{defn}

An absorbing state plays an important role in the theory of the
classical (linear) Markov chains. Analogously, in
\cite{man(2016)fin} it has been introduced the concept of absorbing
sets for nonlinear Markov chains.

\begin{defn}
A subset $ A \subset I_{m} $ is called \textit{absorbing} if $ A^{c}
= \bigcap\limits_{i_{[1,l]}\in A}null(P_{i_{[i,l]},\bullet})$
\end{defn}

It was proven in \cite{man(2016)fin,man(2016)2d} the following
results:
\begin{prop}\label{prop_4 ele}
The following statements hold:
\begin{itemize}
\item[(i)] $ supp(\gb(\xb)) = \bigcup\limits_{i_{[1,l]} \in supp(\xb)}supp(P_{i_{[1,l]},\bullet}) $
\item[(ii)] $ null(\gb(\xb)) = \bigcap\limits_{i_{[1,l]} \in supp(\xb)}null(P_{i_{[1,l]},\bullet}) $
\item[(iii)] $ \gb(int \G_{A}) \subset int \G_{B} \textmd{ where } B=\bigcup\limits_{i_{[1,l]}\in A}supp(P_{i_{[1,l]},\bullet}) $
\item[(iv)] $ \gb(int \G_{A}) \subset int \G_{B} \textmd{ if and only if } \gb(\xb^{(0)}) \in int\G_{B} \textmd{ for some } \xb^{(0)} \in int\G_{A}  $
\end{itemize}
\end{prop}

\begin{prop}\label{prop_eqval}
Let $ A \subset I_{m} $ be a subset. The following statements are equivalent:
\begin{itemize}
\item[(i)] The set $ A $ is absorbing
\item[(ii)] $ \gb(int\G_{A}) \subset int\G_{A} $
\item[(iii)] $ \gb(\xb^{(0)}) \in int\G_{A} \textmd{ for some } \xb^{(0)} \in int\G_{A} $
\end{itemize}
\end{prop}

\section{Surjectivity and Orthogonal Preservness of PSOs}

In this section, we prove the main result of the whole paper.
Namely, we will establish  that the surjectivity of PSO is
equivalent to its OP property. First, we need some auxiliary facts.

\begin{prop}\label{prop_less_3}
    If any subset $ A \subset I_{m} $ with $ \vert A \vert \leq l $ is absorbing, then all subsets of $ I_{m} $ are absorbing. \end{prop}
\begin{proof}
    Using the assumption one can check that for any $ i,j_{0} \in I_{m} $ we have
    \begin{eqnarray}\label{eqn_val_P}
    \left\{
    \begin{array}{ll}
    P_{\ib,\bullet} = {\eb}_{i}, \  P_{i_{[1,l]} \in \{ j_{1}, j_{2} \} ,\bullet} \in conv\{\eb_{j_{1}},\eb_{j_{2}}\},\cdots,\\[2mm]
    P_{i_{[1,l]} \in \{ j_{1},\dots j_{l} \} ,\bullet} \in
    conv\{\eb_{j_{1}},\dots,\eb_{j_{l}}\}\\
    \end{array}
    \right.
    \end{eqnarray}
    where $ j_{k} \neq j_{n} $, $ k,n \in I_{l} $, $\ib=\underbrace{i\cdots
    i}_l$.
    To see this, we consider a set $ A^{(k)} = \{  j_{1},\dots, j_{k} \} $ with $ k \leq l $. Clearly, if  $ n \notin A^{(k)}  $,
    then $ P_{i_{[1,l]} \in \{ j_{[1,k]} \}, n } =0 $. This yields \eqref{eqn_val_P}.

%   If this is not the case, then for some $ k \in I_{m} $ $ P_{i_{[1,l]} \in \{ j_{[1,k]} \}, \bullet } \in conv\{ e_{q_{1}},\dots, e_{q_{n}} \} $ where there exist for some $ t_{0} $ such that $ q_{t_{0}} \neq j_{s} $.
%   This means that
%   \[ P_{i_{[1,l]} \in \{ j_{[1,k]} \}, q_{t_{0}} } > 0 \]
%   On the other hand, one knows that the set $ \{ j_{1},\dots, j_{k} \} $ are absorbing. Clearly $ q_{t_{0}} \notin \{ j_{1},\dots, j_{k} \} $, hence $ P_{i_{[1,l]} \in \{ j_{[1,k]} \}, q_{t_{0}} } =0  $

    Keeping in mind that the coefficients $\{ P_{i_{[1,l+1]}}\} $ satisfy \eqref{eqn_prop_P}, and due to the fact  \eqref{eqn_val_P} one had
     $ null(P_{i_{[1,l]},\bullet}) \supset I_{m}\backslash \{ i_{1},\dots, i_{l} \} $, therefore for any $ B\subset I_{m} $, one finds
    \begin{eqnarray}
    \bigcap\limits_{i_{1},\dots, i_{l} \in B} null(P_{i_{[1,l],\bullet}})&=&  \bigcap\limits_{i \in B} null(P_{\underbrace{i\cdots i}_{l},\bullet})\cap
    \bigcap\limits_{\stackrel{i_{1}\neq i_{2}}{i_{1},i_{2} \in B}} null(P_{ \underbrace{i_{1}\dots i_{1}}_{l-1}i_{2} ,\bullet})\cap \cdots \nonumber \\
    %\bigcap\limits_{\stackrel{i\neq j}{i,j \in B}} null(P_{iji,\bullet})
    %\cap
    %\nonumber \\
    %&& \bigcap\limits_{\stackrel{i\neq j}{i,j \in B}} null(P_{jii,\bullet})
    %\cap
    && \bigcap\limits_{\stackrel{i_{1}\neq \cdots \neq i_{l}}{i_{1}, \dots, i_{l} \in B}} null(P_{i_{[1,l]},\bullet}) \nonumber\\
    &=& I_{m}\backslash B = B^{c} \nonumber
    \end{eqnarray}

    This means that $ B $ is absorbing. This completes the proof.
\end{proof}

\begin{prop} \label{thm_car3_abs_sur}
    If any subset $ A \subset I_{m} $ with $ \vert A \vert \leq l $ is absorbing, then the associated PSO $ \gb:S^{m-1} \rightarrow S^{m-1} $ is surjective.

\end{prop}
\begin{proof}
   According to Propositions \ref{prop_eqval} and \ref{prop_less_3} we infer that the associated PSO $ \gb:S^{m-1} \rightarrow S^{m-1} $ maps each facet of
   the simplex into itself. To show the operator $ \gb $ is surjective, we use mathematical induction with respect to the dimension of the simplex.
    In the case of $ m=2 $, we can write $ \gb $ (see to \eqref{eqn_val_P}) as follows:
    {\tiny
        \begin{eqnarray}
        \gb(\xb)_{1} &=& x_{1}^{l}+\binom l1 P_{1\cdots 12,1} x_{1}^{l-1}x_{2}+ \cdots + \binom {l}{l-(l-k)}P_{\underbrace{\tiny 1\cdots 1}_{l-k} \underbrace{2\cdots 2}_{k} ,1 }x_{1}^{l-k}x_{2}^{k} + \cdots + \binom l1 P_{1\underbrace{\tiny 2\cdots 2}_{l-1} ,1} x_{1}^{}x_{2}^{l-1}, \nonumber \\
        \gb(\xb)_{2} &=& x_{1}^{l}+\binom l1 P_{1\cdots 12,2} x_{1}^{l-1}x_{2}+
        \cdots + \binom {l}{l-(l-k)}P_{\underbrace{\tiny 1\cdots 1}_{l-k} \underbrace{2\cdots 2}_{k} ,2 }x_{1}^{l-k}x_{2}^{k} + \cdots +
        \binom l1 P_{1\underbrace{\tiny 2\cdots 2}_{l-1} ,2} x_{1}^{}x_{2}^{l-1}, \nonumber
        \end{eqnarray}
    }
        where $ \xb = (x_{1},x_{2}) \in S^{1} $. It is enough for us to study $ \gb(\xb)_{1} $ because
        of $ \gb(\xb)_{1}+\gb(\xb)_{2}=1 $.
        Let
        { \tiny
        $$  f(x) = x^{l}+\binom l1 P_{1\cdots 12,1} x^{l-1}(1-x)+ \cdots + \binom {l}{l-(l-k)}P_{\underbrace{\tiny 1\cdots 1}_{l-k} \underbrace{2\cdots 2}_{k} ,1 }x^{l-k}(1-x)^{k} + \cdots + \binom l1 P_{1\underbrace{\tiny 2\cdots 2}_{l-1} ,1} x^{}(1-x)^{l-1}. $$
}
        One can see that  $ f(x) \leq 1  $
        and continuous on interval $ [0,1] $. Due to $ f(0)=0 $ and $ f(1)=1 $, one concludes that $ f(x) $ is surjective over interval $ [0,1] $,
        hence it implies the surjectivity of $ \gb(\xb) $.
        Thus, the statement is true for $ m=2 $. Furthermore, we assume that the statement holds for $ m \leq n-1 $, and we will prove it for $ m=n $.
        From the assumption, if we restrict the mapping of $ \gb $ to the facet, then the mapping is surjective i.e., $ \gb:\partial S^{n-1} \rightarrow \partial S^{n-1} $ is surjective. Now, consider $ \yb \in intS^{n-1} $. Here, surjectivity means that the set $ \gb^{-1}(\yb) $ is nonempty. To prove this statement, we use contradiction by supposing  the  set $ \gb^{-1}(\yb) $ is empty.
    We define a mapping $ g: S^{n-1}\backslash \{\yb\} \rightarrow \partial S^{n-1} $ which maps every point $ \zb \in S^{n-1}\backslash \{\yb\} $ to the intersection point of the ray starting from $ \zb $ in the direction of $ \yb $ with the boundary of the simplex. It is easy to check that the mapping $ \mathcal{F}:S^{n-1} \rightarrow S^{n-1}, \ \mathcal{F} = g\circ \gb $ does not have any fixed point. However, this contradicts to the Brouwer fixed point theorem. This completes the proof.
\end{proof}

\begin{thm}
    Let $ \gb $ be a PSO that maps from $ S^{m-1} $ into itself such that $ \gb(\eb_{i}) = \eb_{i} $ for all $ i \in I_{m} $.
    Then the following statements are equivalent:
    \begin{itemize}
    \item[(i)] $ \gb $ is orthogonal preserving;
    \item[(ii)] $ \gb $ is surjective;
    \item[(iii)] $ \gb $ satisfies the following conditions:
    \begin{itemize}
        \item[(1)] $ \gb^{-1}(\eb_{i}) = \eb_{i} \textmd{ for any } i \in I_{m} $
        \item[(2)] $ \gb^{-1}(int \G_{\eb_{i_{1}}\eb_{i_{2}}})=int \G_{\eb_{i_{1}}\eb_{i_{2}}}  \textmd{ for any } i_{1},i_{2} \in I_{m}, $ \\
        \vdots
        \item[($ l $)] $ \gb^{-1}(int \G_{\eb_{i_{1}} \cdots \eb_{i_{l}}})=int \G_{\eb_{i_{1}} \cdots \eb_{i_{l}}} \textmd{ for any } i,j \in I_{m}, $
    \end{itemize}
    where $  \G_{\eb_{i_{1}}\cdots \eb_{i_{l}}} = conv\{ \eb_{i_{1}},\dots,\eb_{i_{l}} \}
    $.
    \end{itemize}
\end{thm}
\begin{proof}
To prove the theorem,  we will establish the following implications:
$ (i) \Rightarrow (ii) \Rightarrow (iii) \Rightarrow (ii)
\Rightarrow (i)$.

$ (i) \Rightarrow (ii) $. Let $ \gb $ be an orthogonal preserving
PSO. Due to the assumption (i.e., $ \gb(\eb_{i}) = \eb_{i} $), one
has
\[ P_{\underbrace{i\cdots i}_{l},\bullet} = \eb_{i} \]
Now, choose $$ \xb^{(n)} = \left( \dfrac{1}{m-1},
\dots,\dfrac{1}{m-1}, \underbrace{0}_{n^{th}term},\dfrac{1}{m-1},
\dots, \dfrac{1}{m-1} \right)  $$ and $ \eb_{n} $, where $ n \in
I_{m} $. Clearly $ \xb^{(n)} $ is orthogonal to $ \eb_{n} $. Using
the definition of PSO, we have
\begin{eqnarray}
\gb(\xb^{(n)})_{n} &=& \sum\limits_{i_{1}, \dots, i_{l}=1}^{m}P_{i_{[1,l]},n}x_{i_{1}}\cdots x_{i_{l}} \nonumber
 =  \dfrac{1}{(m-1)^{l}} \sum\limits_{\stackrel{ i_{1}, \dots, i_{l}=1 }{ i_{1} \neq n , \dots, i_{l} \neq n }}^{m} P_{i_{[1,l]},n}
%&=& 3 \sum\limits_{\stackrel{j=1}{j\neq l}}^{m}P_{llj,l}x_{l}^{2}x_{j}+ 3 \sum\limits_{\stackrel{i=1}{i\neq l}}^{m}P_{iil,l}x_{i}^{2}x_{l} + 6  \sum\limits_{\stackrel{j=1}{i\neq j\neq l}}^{m}P_{ijl,l}x_{i}x_{j}x_{l} \nonumber
%& = & \dfrac{ \binom l1 }{(m-1)^{l}}\sum\limits_{\stackrel{i_{1},i_{2}=1}{i_{1}\neq i_{2}\neq n}}^{m}P_{i_{1}\cdots i_{1}i_{2},n} + \dots  +
%\dfrac{ \binom l1 }{(m-1)^{l}}\sum\limits_{\stackrel{i_{1},i_{2}=1}{i_{1}\neq i_{2}\neq n}}^{m}P_{\underbrace{i_{1} \cdots i_{1} }_{s}\underbrace{i_{2}\cdot i_{2}}_{t},n} + \nonumber \\
%&& \dfrac{ \binom l1 }{(m-1)^{l}}\sum\limits_{\stackrel{i_{1},i_{2}, i_{3}=1}{i_{1}\neq i_{2} \neq i_{3}\neq n}}^{m}P_{i_{1}\cdots i_{1}i_{2}i_{3},n}
\end{eqnarray}
From the orthogonal preservness of $ \gb $, we infer that $
\gb(\xb^{(n)}) $ is orthogonal to $ \gb(\eb_{n}) $, whence
\[ P_{i_{[1,l]},n} =0 \textmd{ if } i_{1} \neq n , \dots, i_{l} \neq n \]
This yields \eqref{eqn_val_P}, therefore all subsets $ A \subset
I_{m} $ with $ \vert A \vert \leq l $ are absorbing. According to
Proposition \ref{thm_car3_abs_sur} we obtain that $ \gb $ is
surjective.

$ (ii) \Rightarrow (iii) $. Assume that $ \gb $ is surjective and
let $ \gb^{-1}(\eb_i) $ be the preimage of $ \eb_{i} $. We set
\[ supp(\gb^{-1}(\eb_i)) = \bigcup\limits_{\xb \in \gb^{-1}(\eb_i)}supp(\xb), \ \ \G_{supp(\gb^{-1}(\eb_i))} = conv\{ \eb_{j} \}_{j \in supp(\gb^{-1}(\eb_i))}  \]

Due to Proposition \ref{prop_4 ele} one gets $
\gb(\G_{supp(\gb^{-1}(\eb_{i}))}) =\eb_{i} $. Consequently,
\[ \{\eb_{j}\}_{j \in supp(\gb^{-1}(\eb_{i}))} \subset \gb^{-1}(\eb_{i}) \ \textmd{ for any } i\in I_{m}. \]
It implies that $ |supp(\gb^{-1}(\eb_{i})) =1 | $, which means that
only $ \eb_{i} $ that maps to $ \eb_{i} $ hence we obtain (iii)(1).

Further, let $ k \in \{ 2,\dots,l \} $. Take $ \yb \in
int\G_{\eb_{i_{1}}\cdots \eb_{i_{k}}} $ and let $ \xb \in
\gb^{-1}(\yb) $. Using Proposition \ref{prop_4 ele} we have
\[ \gb(int\G_{supp(\xb)} ) \subset int\G_{\eb_{i_{1}} \cdots \eb_{i_{k}}} \]
In fact, we have
\[ supp(\xb) = \{ i_{1},\dots, i_{k} \} \ \textmd{ for any } \xb \in V^{-1}(\yb). \]
If not, then there exists $ k' \in supp(\xb) \backslash \{
i_{1},\dots, i_{k} \} \neq \emptyset $. Then $ V(\eb_{k'}) \in
V(int\G_{supp(\xb)}) \subset \G_{\eb_{i_{1}}\cdots \eb_{i_{k}}} $,
which is a contradiction. Therefore we obtain (iii)($k$) for any $ k
\in \{ 2,\dots,l \} $.

$ (iii) \Rightarrow (ii) $. This implication immediately follows
from Proposition \ref{thm_car3_abs_sur}.

$ (ii) \Rightarrow (i) $ Due to the surjectivity of $ \gb $  and
condition $ \gb(\eb_{i}) = \eb_{i} $ one gets any subset $ A \subset
I_{m} $ with $ \vert A \vert \leq l $ is absorbing. It follows from
\eqref{eqn_val_P} that
\begin{eqnarray}\label{BB}
\gb(\xb)_{k} &=& \sum\limits_{i_{1},\dots,i_{k} = 1}^{m}P_{i_{[1,l]},k}x_{i_{1}}\cdots x_{i_{l}} =
x_{k} \left( \sum\limits_{i_{2},\dots,i_{l} = 1}^{m}P_{ki_{[2,l]},k}x_{i_{2}} \cdots x_{i_{l}}\right)  %\\
%&=& P_{lll,l}x_{l}^{3}+3 \sum\limits_{\stackrel{j=1}{j\neq l}}^{m}P_{llj,l}x_{l}^{2}x_{j}+ 3 \sum\limits_{\stackrel{i=1}{i\neq l}}^{m}P_{iil,l}x_{i}^{2}x_{l} + 6  \sum\limits_{\stackrel{j=1}{i\neq j\neq l}}^{m}P_{ijl,l}x_{i}x_{j}x_{l} \nonumber \\
%&=& x_{l} \left(   x_{l}^{2}+3 \sum\limits_{\stackrel{j=1}{j\neq l}}^{m}P_{llj,l}x_{l}x_{j}+ 3 \sum\limits_{\stackrel{i=1}{i\neq l}}^{m}P_{iil,l}x_{i}^{2} + 6  \sum\limits_{\stackrel{j=1}{i\neq j\neq l}}^{m}P_{ijl,l}x_{i}x_{j} \right)
\end{eqnarray}
for any $ \xb \in S^{m-1} $ and $ k \in I_{m} $. Next, take any two
orthogonal vectors  $ \xb = (x_{1},\dots,x_{m})$ and $\yb
=(y_{1},\dots,y_{m}) $ in the simplex $ S^{m-1} $. This means that
for any fix $ k \in I_{m} $ $ x_{k}=0 $ or $ y_{k}=0 $. Therefore,
from \eqref{BB} we infer that $ \gb(\xb)_{k}\cdot \gb(\yb)_{k} = 0 $
which yields the orthogonality $ \gb(\xb) $ and $ \gb(\yb) $. This
completes the proof.
\end{proof}

Immediately, from the last theorem one concludes the following
corollary.

\begin{cor}
    Let $ \gb $ be a CSO that maps from $ S^{m-1} $ into itself. Then the following statements are
    equivalent:
    \begin{itemize}
        \item[(i)] $ \gb $ is orthogonal preserving;
        \item[(ii)] $ \gb $ is surjective;
        \item[(iii)] $ \gb $ satisfy the following conditions:
        \begin{itemize}
            \item[(1)] $ \gb^{-1}(\eb_{i}) = \eb_{\pi(i)} \textmd{ for any } i \in I_{m}
            $,
            \item[(2)] $ \gb^{-1}(int \G_{\eb_{i_{1}}\eb_{i_{2}}})=int \G_{\eb_{\pi(i_{1})}\eb_{\pi(i_{2})}}  \textmd{ for any } i_{1},i_{2} \in I_{m}, $ \\
            \vdots
            \item[($ l $)] $ \gb^{-1}(int \G_{\eb_{i_{1}} \cdots \eb_{i_{l}}})=int \G_{\eb_{\pi(i_{1})} \cdots \eb_{\pi(i_{l})}} \textmd{ for any } i,j \in I_{m}, $
        \end{itemize}
        for some permutation $\pi$ of $I_m$.
    \end{itemize}
\end{cor}

\begin{rem} It is known \cite{man(2016)2d} that if $V$ is a surjective QSO, then it is
bijection.  Therefore, we formulate the following conjecture.
\end{rem}

\begin{conj} Any surjective PSO is bijection.
\end{conj}

We point out that some sufficient conditions for the bijectivity of
PSO of the form \eqref{BB} has been provided in \cite{MS-new}.

\end{document}